\documentclass[a4paper,12pt]{article}
  \usepackage{amsmath,amsfonts,amssymb,amsthm}

\usepackage[active]{srcltx} 
\usepackage[english]{babel}

\usepackage{color,soul}

\newtheorem{theorem}{Theorem}[section]

\newtheorem{lemma}[theorem]{Lemma}
\newtheorem{proposition}[theorem]{Proposition}

\theoremstyle{definition}
\newtheorem{definition}[theorem]{Definition}
\newtheorem{remark}{Remark}

\def\N{\mathbb{N}}

\def\R{\mathbb{R}}

\let\e=\varepsilon
\let\vp=\varphi
\let\vt=\vartheta
\let\t=\tilde
\let\ol=\overline

\let\.=\cdot
\let\0=\emptyset
\let\mc=\mathcal

\def\solose{\ \Rightarrow\ }

\def\O{\Omega}
\def\oo{\ol\O}

\def\po{\partial \O}
\def\la{\lambda}
\def\vep{\varepsilon}
\def\de{\delta}
\def\vfi{\varphi}

\def\dv{\text{\rm div}}

\def\Tr{\text{\rm Tr}}

\newcommand{\su}[2]{\genfrac{}{}{0pt}{}{#1}{#2}}

\def\eq#1{{\rm(\ref{eq:#1})}}
\def\thm#1{Theorem \ref{thm:#1}}
\def\square{\hbox{$\sqcap\kern-7pt\sqcup$}}

\def\seq#1{(#1_n)_{n\in\N}}

\def\pe{principal eigenvalue}

\def\l{\lambda_1}

\newenvironment{formula}[1]{\begin{equation}\label{eq:#1}}
                       {\end{equation}\noindent}
\def\Fi#1{\begin{formula}{#1}}
\def\Ff{\end{formula}\noindent}

\def\pe{principal eigenvalue}
\def\MP{\hbox{\bf MP}}

\def\m{\mu_1(F,\O)}

\makeatletter
\let\@fnsymbol\@alph
\makeatother

\title{
\bf{Maximum Principle and generalized principal eigenvalue for 
degenerate elliptic operators}}
\author{Henri Berestycki \thanks{Ecole des Hautes Etudes en Sciences Sociales,
CAMS, 190-198, av. de France, 75244 Paris, France}
\and Italo Capuzzo Dolcetta \thanks{Dipartimento di Matematica, Sapienza
Universit\`a di Roma, piazzale A. Moro 3, 00185 Roma, Italy}
\and Alessio Porretta \thanks{Dipartimento di Matematica, Universit\`a di Roma Tor Vergata, Via della Ricerca Scientifica 1, 00133 Roma, Italy.}
\and Luca Rossi \thanks{ Dipartimento di
Matematica, Universit\`a di Padova, via Trieste 63, 35121 Padova,
Italy}}

\begin{document}

\maketitle

\begin{abstract}

We characterize the validity of the Maximum Principle in bounded domains for fully nonlinear degenerate elliptic operators  in terms of the sign of a suitably defined generalized \pe. Here, maximum 
principle refers to the non-positivity of viscosity subsolutions  of the Dirichlet problem. This characterization is derived in terms of a new notion of generalized \pe, which is needed because of the possible degeneracy of the operator, admitted in full generality. We further discuss the relations between this notion and other natural generalizations of the classical notion of \pe, some of which had already 
been used in the literature for particular classes of 
operators.
%
\end{abstract}


\section{Introduction}\label{sec:intro}

This paper is concerned with the Maximum Principle property for degenerate
second order elliptic
operators. Our aim is to characterize the validity of the Maximum Principle for
arbitrary degeneracy of the operator - including the limiting cases of first
and zero-order operators - in terms of the sign of a suitably defined 
generalized \pe. 
Such a complete characterization   is missing, as far as we know, even for the
case of linear operators, which was of course our first motivation. Due to the possible loss of regularity, as well as of boundary conditions, which is caused by degeneracy of ellipticity, the appropriate framework to deal with  this problem is, even in 
the linear case, that of viscosity solutions.
This approach is of course not restricted to the linear case, so we study the question in the 
more general setting of homogeneous fully nonlinear degenerate elliptic operators $F(x,u,Du,D^2u)$.
\\

%
Let $\O$ be a bounded domain in $\R^N$ and $\mathcal{S}_N$ be the
space of $n\times n$ symmetric matrices endowed with the usual partial order,
with $I$ being the identity matrix.
A fully nonlinear operator $F: \O\times\R\times \R^N\times \mathcal{S}_N\to
\R$ is said to be 
%
{\em degenerate elliptic} if 
 $F$ is non 
increasing in the matrix entry, see 
condition (H1) in the next section. The basic example to have in mind is that
of linear operators in non divergence form
$$F(x,u,Du,D^2u)=-\Tr(A(x)D^2u)-b(x)\.Du-c(x)u,\quad x\in\O,$$
where $A(x)$ is nonnegative definite.
\\

We are interested
in the following version of the Maximum Principle, \MP $\,$  in short :
\begin{definition}\label{def:vMP}
The operator $F$ satisfies \MP\ in $\O$ if every viscosity subsolution
$u\in USC(\ol\O)$ of the Dirichlet problem
\Fi{DP}
\left\{\begin{array}{ll}
F(x,u,Du,D^2u)=0 & \text{in }\O\\
u=0 & \text{on }\partial\O,
\end{array}\right.
\Ff
satisfies $u\leq0$ in $\ol\O$.
\end{definition}
We denote by $USC(\ol\O)$ the set of upper semicontinuous functions on $\ol\O$.
It is worth pointing out that in the above definition both the PDE and the
boundary conditions are understood in the viscosity sense (see
Section 7 of \cite{user}). Precisely, $u$ is a subsolution of \eq{DP} if for all
$\vp\in C^2(\ol\O)$ and $\xi\in\ol\O$ such that 
$(u-\vp)(\xi)=\max_{\ol\O}(u-\vp)$, it holds that
$$F(\xi,u(\xi),D\vp(\xi),D^2\vp(\xi)) \leq0\quad\text{if }\xi\in\O,$$
$$\max\big[u(\xi)\,,\, F(\xi,u(\xi),D\vp(\xi),D^2\vp(\xi))\big] \leq
0\quad\text{if }\xi\in\partial\O.$$
\\ Note, in particular, that the validity of the \MP\ property
implies that viscosity subsolutions cannot be positive on $\partial\O$, namely,
the inequality $u\leq0$ on $\partial\O$ holds in the classical pointwise sense.
\\

Before describing our results, let us recall some
classical and more recent results concerning the Maximum Principle and the \pe.

A standard result in the viscosity theory is that, under suitable
continuity assumptions on the degenerate elliptic operator $F$, the Maximum
Principle for viscosity 
subsolutions holds true if $r\mapsto F(x,r,p,X)$ is strictly increasing (see
e.g.~\cite{user}). This is only a sufficient condition.
It is well known that if $\O$ is a bounded smooth domain and $F$ is a uniformly
elliptic linear operator with smooth coefficients, then the validity of the
Maximum Principle for classical subsolutions is equivalent to the positivity of
the \pe\ $\l(F,\O)$ associated with Dirichlet boundary condition. This
eigenvalue is the bottom of the spectrum
of the operator $F$ acting on functions satisfying the Dirichlet boundary
conditions. It follows from the Krein-Rutman
theory that $\l(F,\O)$ is simple and the associated eigenfunction $\vp$ is 
positive
in $\O$.  So, if $\l(F,\O)\leq0$ then $\vp$ violates the Maximum Principle.
As a consequence, if the Maximum Principle holds then the problem admits a
positive strict supersolution. \\
The reverse implication is also true, but its
proof is
not completely straightforward since it requires an analysis of how sub and 
supersolutions can vanish at the boundary. To this aim, one typically makes 
use
of barriers and Hopf's lemma, for which uniform ellipticity or some other 
properties are required.\\ As
we will see, the possibility of different behaviours of supersolutions at the
boundary is one of the most delicate points one has to handle in order to
deal with degenerate operators.
\\

The connection between the Maximum Principle and the existence
of positive strict supersolutions led the first author, L.~Nirenberg and
S.~R.~S.~Varadhan to introduce in \cite{BNV} the following notion of {\em
generalized \pe}:
$$
\l(F,\O):=\sup\{\lambda\in\R\ :\ \exists \phi\in W^{2,N}_{loc}(\O),\
\phi>0\text{ in }\O,\ F[\phi]-\lambda\phi\geq0\text{ a.e.~in }\O\}.
$$
Here and henceforth we will write $F[\.](x)$ for short, or simply $F[\.]$, in 
place of
$F(x,\.,D\.,D^2\.)$.\\
Using this generalization, they were able to extend the characterization of the
Maximum Principle  for linear elliptic operators to the case of non-smooth domains,
where the classical \pe\ is not defined.

In \cite{BD07}, I.~Birindelli and F.~Demengel adapted the definition of
\cite{BNV} to  a class of fully
nonlinear operators $F$ which are homogeneous of degree $\alpha>0$,  including some degenerate elliptic operators which are modeled on the example of the $p-$Laplacian.  They defined the principal eigenvalue as 
$$
\l(F,\O):=\sup\{\lambda\in\R\ :\ \exists \phi\in LSC(\O),\ \phi>0\text{ and } 
F[\phi]-\lambda\phi^\alpha\geq0\text{ in }\O\}.
$$
Here, $LSC(\O)$ denotes the set of lower semicontinuous functions on $\O$.

Actually, in their earlier
work \cite{BD06}, the same authors had  defined the generalized \pe\ in the following slightly
different
way:
$$
\ol\lambda_1(F,\O):=\sup\{\lambda\in\R\ :\ \phi\in LSC(\O),\ \inf_\O\phi>0,\
F[\phi]-\lambda\phi^\alpha\geq0\text{ in }\O\}.
$$
The two notions coincide in the case treated in \cite{BD07}, but, as we will 
see in the proof of Proposition \ref{pro:ce} part (i) below,
this may not be the case in general. Let us mention that the non-equivalence
between $\l$ and $\ol\lambda_1$, which in the cases considered in the present
paper is due to the degeneracy of the operator, can occur when $\O$ is unbounded
even for uniformly elliptic linear operators. The characterization of the
Maximum Principle in terms of generalized \pe s such as $\l$ and $\ol\lambda_1$,
as well as the study of their properties, for uniformly elliptic linear
operators in unbounded domains is the object of the recent paper \cite{BR4}.
\\

It turns out that the generalized \pe\ $\l$ is not well suited
to characterize the validity of  \MP \, for the general degenerate cases that
we address in the present paper. This is showed in the next section, 
in which we also discuss the pertinence of other natural candidates, such as
$\ol\lambda_1$,  as well as the limit of the \pe s of the $\e$-viscosity
regularized 
operators $F^\e = -\e \Delta+F$.   None of those choices will be sufficient in order to characterize the \MP\  property  in the most general situation. Indeed, one of our main goals is to identify the right set of
admissible functions, in the definition of a generalized \pe, which can be suitable for 
general degenerate elliptic operators. Eventually, the right notion for our purposes turns out to be given by the following.

\begin{definition}
Given a domain $\O$ in $\R^N$ and an open set $\mc{O}$ such that
$\ol\O\subset\mc{O}$, and a fully nonlinear degenerate elliptic operator $F$ in
$\mc{O}$, we define
$$
\mu_1(F,\O):=\sup\{\lambda\in\R\ :\ \exists\O'\supset\ol\O,\ \phi\in
LSC(\O'),\ \phi>0\text{ and } 
F[\phi]-\lambda\phi^\alpha\geq0\text{ in }\O'\}.
$$
\end{definition}

 Hence, the definition of the generalized \pe\ $\mu_1$ in a domain
$\O$ requires the operator to be defined in a larger set.
Equivalent formulations for $\m$ are
$$\mu_1(F,\O)=\sup_{\O'\supset\ol\O}\l(F,\O')=\lim_{\e\to0^+}\l(F,\O+B_\e),$$
where the last one follows from the monotonicity of $\l(F,\O)$ with respect to
inclusion of the domains.
%
%

\subsection{Hypotheses and main result}

Throughout the paper, $\O$ is a bounded domain in $\R^N$, not necessarily smooth, and $\mc{O}$ is an open
set such that
$\ol\O\subset\mc{O}\subset \R^N$.
We assume that  $F:\mc{O}\times\R\times\R^N\times\mc{S}_N\to\R$ is a continuous function which satisfies the following hypotheses:
\begin{itemize}
 \item[(H1)]\label{elldeg}
$F(x,r,p,X+Y)-F(x,r,p,X)\leq0,\quad\forall
(x,r,p,X,Y)\in\mc{O}\times\R\times\R^N\times\mc{S}_N\times\mc{S}_N,\ Y\geq0$;
 \item[(H2)]\label{hom}
$\exists\alpha>0,\ F(x,\tau r,\tau p,\tau X)=\tau^\alpha F(x,r,p,X),
\quad\forall (x,r,p,X)\in\mc{O}\times\R\times\R^N\times\mc{S}_N,\ \tau\geq0$;
 \item[(H3)]\label{cont}
$r\mapsto F(x,r,p,X)$ is continuous, uniformly with respect to
$(x,p,X)\in\mc{O}\times\R^N\times\mc{S}_N$;
 \item[(H4)]\label{somma}
For all $R>0$, there exists a function $\omega\in C([0,+\infty))$ with
$\omega(0)=0$ such that
if $X,Y\in\mc{S}_N$ satisfy
$$\exists\sigma>0,\quad
-3\sigma\left(\begin{matrix}
               I & 0 \\
0 & I
              \end{matrix}
\right)
\leq\left(\begin{matrix}
               X & 0 \\
0 & -Y
              \end{matrix}
\right)\leq3\sigma\left(\begin{matrix}
               I & -I \\
-I & I
              \end{matrix}
\right),$$
then
$$F(y,r,\sigma(x-y),Y)-F(x,r,\sigma(x-y),X)\leq\omega(\sigma|x-y|^2+|x-y|),
\quad\forall x,y\in\mc{O},\ |r|\leq R.$$
\end{itemize}

As 
it was established in \cite{user}, hypothesis (H4) is the key structure
condition for the validity of the Comparison Theorem for viscosity solutions of
degenerate elliptic equations. Let us emphasize that no regularity
assumption is required on the set $\O$.
\\

We now state the main result of this article.
\begin{theorem}\label{thm:main}
Under the assumptions (H1)-(H4), $F$ satisfies the \MP\ property in
$\O\subset\subset\mc{O}$ if and only if $\mu_1(F,\O)>0$.
\end{theorem}

Some remarks on the statement of Theorem \ref{thm:main} are in order. Since our
characterization of \MP\ in $\O$ requires the operator $F$
to be defined in some $\mc{O}\supset\ol\O$ then, if $F$ is just defined in $\O$
and satisfies (H1)-(H4) there, in order to apply our result we need to
extend it to an operator satisfying (H1)-(H4) in the larger domain $\mc{O}$. 
This is not completely satisfactory, even though the result itself ensures that
the notion $\mu_1(F,\O)$ does not depend on the
particular extension. A characterization
expressed in terms of a more intrinsic notion, such as $\l$ or
$\ol\lambda_1$, would be preferable.
Proposition \ref{pro:ce} below provides examples showing that \MP\ is not
guaranteed by $\l>0$ nor by $\ol\lambda_1>0$. However, in the
case of
$\ol\lambda_1$, the only examples we are able to construct do not satisfy
(H4).

We leave it as an open problem to know whether $\mu_1$ coincides with
$\ol\lambda_1$ under the assumption (H4), and then whether $\mu_1(F,\O)$ can be
replaced by 
$\ol\lambda_1(F,\O)$ in \thm{main}.
In Theorems \ref{thm:l''=mu} and \ref{thm:l''=mulinear} below we show that, for a  smooth domain $\Omega$, 
this is true in two significant cases: if the operator  admits barriers at each point of the
boundary, and, for the case of linear operators, if
in each connected component of $\partial\O$ the so--called Fichera condition is
either always satisfied or always violated. The general case remains open.

Finally, let us point out that a generalized principal eigenfunction associated
with $\mu_1$ does not always exist (see Remark \ref{rem:pef} in Section
\ref{sec:proof}). This is due to the degeneracy of the operator and is true
even in the linear case.


\subsection{Examples}

In this section we present some examples of operators to which \thm{main} 
applies, recovering some known results. We further analyse the generalized \pe s
in these
particular cases. The examples are divided into classes, but none of them is
intended to be exhaustive.
\\

\noindent{\em The standard sufficient condition}

If an operator $F$ satisfies $\min_{x\in\ol\O}F(x,r,0,0)>0$ for all $r>0$, then
  \MP\ holds. This is an immediate consequence of the
definition of viscosity subsolution. Notice that in this case
$\ol\lambda_1(F,\O)>0$ and, up to extending $F$ outside $\O$ as a continuous
function, $\mu_1(F,\O)>0$ too.\\

\noindent{\em First-order operators}

\thm{main} applies to the generalized eikonal operator $F[u]=-b(x)|Du|-c(x)u$,
provided that $b\in W^{1,\infty}(\mc{O})$ and $c\in
C(\mc{O})$. The Lipschitz-continuity of $b$ is required for (H4)
to hold. Furthermore, the result still holds for an operator
$G[u]=F[u]-\Tr(A(x)D^2u)$, with $A=\Sigma^t\Sigma$ and
$\Sigma\in W^{1,\infty}(\mc{O})$. 

Another family of operators which can be considered is
$F[u]=-b(x)\.Du|Du|^{\alpha-1}-c(x)|u|^{\alpha-1}u$, with $\alpha\geq1$. The
hypothesis on $b$ is again $b\in W^{1,\infty}(\mc{O})$ if $\alpha=1$, otherwise
we need $(b(x)-b(y))\.(x-y)\leq0$ for $x,y\in\mc{O}$.
\\

\noindent{\em Subelliptic operators}

For several classes of subelliptic operators one can derive the sign of $\mu_1$
and thus apply \thm{main}. 
For example, if the ellipticity of $F$ is not degenerate in a direction $\xi$,
in the sense
that there exists $\beta>0$ such that
$$F(x,r,p,X+\xi\otimes\xi)-F(x,r,p,X)\leq-\beta,\quad\forall
x\in\mc{O},\ r\in\R,\ p\in\R^N,\ X\in\mc{S}_N,$$
and the positive constants are supersolutions of $F=0$ in $\mc{O}$, i.e.,
$F(x,1,0,0)\geq0$ in $\mc{O}$, then $\m>0$.
This is seen by taking $\phi(x)=1-\e e^{\sigma\xi\.x}$, with $\sigma$ large and
then $\e$ small.
\\
The above conditions are satisfied for instance by  the Grushin operator: 
$-\partial_{xx}-|x|^\alpha\partial_{yy}$, with $\alpha>0$. This operator, which 
is a H\"ormander operator if $\alpha$ is an even integer, belongs to the class 
of $\Delta_\lambda$ operators studied in \cite{FL83}.
The key property of such operators is the $2$-homogeneity with respect to a
group of dilations of $\R^N$. Concerning the generalized \pe s
introduced before, this 
property yields $\mu_1(\Delta_\lambda,\O)=\l(\Delta_\lambda,\O)$ if 
$\O$ is starshaped with respect to the origin. Actually, the following much 
weaker scaling property is required on an operator $F$ in order to have 
$\mu_1(F,\O)=\l(F,\O)$: there is a family of $C^2$-diffeomorphisms 
$(\vt_t)_{t>0}$ of $\R^N$ into itself and a function $\psi:\R_+\to\R$
satisfying 
$$
\forall t>1,\quad \ol\O\subset\vt_t(\O),\qquad
\lim_{t\to1}\psi(t)=1,$$
and
$$\forall u\in C^2(\vt_t(\O)),\quad
F[u\circ\vt_t](x)=\psi(t)F[u](\vt_t(x)),\quad x\in\O.$$
Indeed for $\phi\in LSC(\O)$, the above condition implies (in the viscosity 
sense)
$$F[\phi\circ\vt_t^{-1}](x)=(\psi(t))^{-1}F[\phi](\vt_t^{-1}(x)),
\quad x\in\vt_t(\O).$$
It follows from the definition of $\l$ that the mapping 
$t\mapsto\l(F,\vt_t(\O))$ is lower semicontinuous at $t=1$. Hence, 
since for $t>1$, $\l(F,\vt_t(\O))\leq\mu_1(F,\O)\leq\l(F,\O)$, we infer that 
$\mu_1(F,\O)=\l(F,\O)$.
\\

\noindent{\em Parabolic operators} 

It is well known that the classical Maximum
Principle holds for uniformly parabolic linear operators  of the type
$\partial_tu-\Tr(A(t,x)D^2u)-b(t,x)\.Du-c(t,x)u$, with $t>0$,
$x\in\O$. Uniformly parabolic means that
$A\geq\alpha I$, for some positive constant $\alpha$. 
Note that a crucial difference with the elliptic case is that the
Maximum Principle holds
even if the zero order term $c$ is positive and very large. 
One can interpret the validity of the
Maximum Principle as a consequence of the positivity of the \pe s. Indeed,
considering the function $\phi=e^{\sigma t}$ and letting $\sigma\to+\infty$, one
finds that in this case all notions of \pe s introduced in Section 
\ref{sec:intro} are equal to $+\infty$.  However, the  
parabolic Maximum Principle cannot be derived right away from \thm{main} due to 
the unboundedness of the domain.
%
%
%
%
%
%
%
%
\\

\noindent{\em $1$-homogeneous, uniformly elliptic operators}

A fully nonlinear operator $F: \O\times\R\times \R^N\times \mathcal{S}_N\to
\R$ is said to be  {\em uniformly elliptic} if there exists $\alpha>0$ such that
$$
F(x,r,p,X+Y)-F(x,r,p,X)\leq-\alpha {\rm Tr}(Y),\quad\forall
(x,r,p,X,Y)\in \Omega\times\R\times\R^N\times\mc{S}_N\times\mc{S}_N.
$$
An important role in the theory of fully nonlinear, uniformly elliptic operators
is played by $1$-homogeneous
operators, that is, operators satisfying (H2) with $\alpha=1$. Besides linear
operators, this class includes Pucci, Bellman and Isaacs operators.
The latter class, which is the most general, fulfils (H1)-(H4) under 
suitable regularity conditions on the coefficients. Being uniformly elliptic, 
it also satisfies the hypotheses of \thm{l''=mu} below, which implies that the 
\MP\ property is characterized by the sign of $\ol\lambda_1$ if $\O$ is smooth.

Several works have already addressed the question of the validity of the
Maximum Principle and the existence of the principal
eigenfunction. Among them, let us cite the papers 
\cite{FQ} and   \cite{Q} for the Pucci operator 
and   \cite{QS} for more general operators, including
the Bellman one.
In \cite{QS}, the simplicity of the \pe\ is further obtained.
The method used in the above mentioned papers differs from the one of
\cite{BD06} and ours. 
It follows the line of the classical proof based on the Krein-Rutman theory.
This is possible because of the uniform ellipticity of the operator, which makes
the $W^{2,N}$ estimates available and avoids the direct use of
the definition of viscosity solution.
\\

\noindent{\em $1$-homogeneous, degenerate elliptic operators}

Two examples of fully nonlinear degenerate operators are
$$
-\mathcal{P}_k (D^2u) := -\eta_{N-k+1}(D^2u) -\ldots -\eta_N(D^2u), 
$$
$k$ being an integer between $1$ and $N$ and $\eta_1(D^2u) \leq \eta_2 (D^2u) 
\leq \ldots \leq \eta_N(D^2u)$ being the ordered 
eigenvalues of the matrix $D^2u$, and the degenerate maximal Pucci 
operator 
$$
-\mathcal{M}^+_{0,1} (D^2u) := -\sum_{i=1}^N\max(\eta_i(D^2u),0) = 
-\sup_{A\in \mathcal{S}_N,\ 0\leq A\leq I} \Tr(AD^2u).
$$
The operator $-\mathcal{P}_N$ has been used by F.~R. Harvey and H.~B. Lawson
to characterize the validity of the Maximum Principle for operators
only depending on the Hessian: Theorem 2.1 of \cite{HL-SMP} states, with a
geometrical terminology, that an operator
$F:\mc{S}_N\to\R$ satisfies the \MP\ if and only if $F(X)\leq0\solose
-\mathcal{P}_N(X)\leq0$. Notice that if $F$ satisfies such a property then
the function $\phi(x):=k-|x|^2$, with $k>\sup\{|x|^2\ :\ x\in\O\}$, 
satisfies $\phi>0$ in $\ol\O$ and $F(D^2\phi)>0$, whence $\m>0$.

Both $\mathcal{P}_k$ and $-\mathcal{M}^+_{0,1}$ have positive principal
eigenvalue $\mu_1$ and satisfy the \MP. In addition, they admit continuous 
barriers at every point of the boundary 
of a smooth domain, in the sense of Definition \ref{def:barrier} below. 
Therefore \thm{l''=mu} implies that $\mu_1$ coincides with $\ol\lambda_1$.
\\

\noindent{\em The $p$ and the infinity Laplacian}

The $p$ and the infinity Laplacian are defined respectively by
$$\Delta_p u:=\dv(|Du|^{p-2} Du),\quad p>1,\qquad
\Delta_\infty u:=\frac{Du}{|Du|} D^2u\frac{Du}{|Du|}.$$
These definitions have a meaning, in the viscosity sense, if the gradient
is nonzero. One has to extend them in suitable way to get a general definition.
Both the operators $F=-\Delta_p,-\Delta_\infty$, with the possible addition of a
degenerate elliptic
operator sharing the same homogeneity property, fit with
the hypotheses (H1)-(H4) of \thm{main} above.

The characterization of the Maximum Principle for the
$p$-Laplacian was derived by I.~Birindelli and F.~Demengel in
\cite{BD06} and \cite{BD07}, using the \pe\ $\ol\lambda_1$ and $\l$
respectively.
The fact that $\l=\ol\lambda_1=\mu_1$ in that case is due to the validity of
the Hopf
lemma and the existence of barriers. 
The result for the infinity Laplacian is due to P.~Juutinen \cite{infinity} and
is expressed in terms of $\ol\lambda_1$.
The existence of barrier is crucial also in this case.
We remark that, owing to the particular structure of the infinity-Laplacian, 
barriers exist without assuming any regularity of $\partial\O$.
Finally, the existence of a generalized principal eigenfunction is proved in
\cite{BD07} and \cite{infinity}, but not its simplicity.

%
%
%


\section{Exploring other notions of generalized \pe}

In this section we show that the validity of the
\MP\ is not characterized by the positivity of $\l$, nor by that of other
natural notions of generalized \pe.
One is the quantity $\ol\lambda_1(F,\O)$ defined before. 
Another natural candidate is 
$$
\lambda_*(F,\O):=\liminf_{\e\to0^+}\lambda^\e,
$$
where $\lambda^\e$ denotes the classical Dirichlet \pe\ of the regularized
operator $-\e\Delta+F$ in $\O$.
If $\O$ is smooth and $F$ is a uniformly elliptic linear operator with smooth
coefficients, then the notions $\l$, $\ol\lambda_1$, $\mu_1$, $\lambda_*$
coincide. In
the general case we only have that $\mu_1\leq\ol\lambda_1\leq\l$.

We now show that, even in the linear case, the
sign of $\l$, $\lambda_*$, $\ol\lambda_1$ do not characterize the validity of
the \MP\
for degenerate elliptic operators.

\begin{proposition}\label{pro:ce}
For each of the following conditions:
\begin{itemize}
 \item[{\em (i)}] $\l(F,\O)>0$,
 \item[{\em (ii)}] $\lambda_*(F,\O)>0$,
 \item[{\em (iii)}] $\ol\lambda_1(F,\O)>0$,
\end{itemize}
there exists a degenerate elliptic linear operator $F$ with smooth coefficients
in $\O$ that does not satisfies the \MP\ property and yet satisfies that
condition. Moreover, for cases (i) and (ii), such operator satisfies (H4).
\end{proposition}

\begin{proof}
(i) Let $F[u]=\frac x2 u'-u$ and $\O=(0,1)$. The function $u(x)=x(1-x)$ violates
the \MP, but $\l(F,\O)=+\infty$ (as it is seen by taking $\phi(x)=x^n$ in the
definition, with $n\to+\infty$). We remark that in
this case $\ol\lambda_1(F,\O)=\m\leq0$ by Theorems
\ref{thm:main} and \ref{thm:l''=mulinear} below.

(ii) Let $F[u]=-2xu'$ and $\O=(0,1)$. For $\e>0$ and $\phi\in C^2(\O)$, we have
that
$$-\e\phi''+\left(\frac{x^2}\e+1\right)\phi=e^{\frac{x^2}{2\e}}(-\e D^2+F)
\left[\phi e^{-\frac{x^2}{2\e}}\right].$$
As a consequence, $\lambda^\e$ coincides with the Dirichlet \pe\ of the operator
$-\e u''+\left(\frac{x^2}\e+1\right)u$, which is greater than or equal to 1.
This shows that $\lambda_*(F,\O)\geq1$. On the other hand, the indicator
function of $\{0\}$ violates \MP, because, as one can readily check, any smooth
function $\vp$ touching it from above at some $x_0\in[0,1)$ satisfies
$F[\vp](x_0)=0$.

(iii) For this case, we give two examples, one with a first order and one with a
second order operator. The operator
$F[u]=-\sqrt{x}u'$ does not satisfies \MP\ in $\O=(0,1)$, as it is seen by
taking $u$ equal to the indicator function of $\{0\}$. But, taking
$\phi(x)=2-\sqrt{x}$ in the
definition of $\ol\lambda_1$ yields $\ol\lambda_1(F,\O)\geq1/4$. 

An example of the second order is provided by $F[u]=-xu''$ and $\O=(0,1)$. As
before, 
the indicator function of $\{0\}$ violates \MP. On the other hand, the
function $\phi(x)=1+\sqrt x$ satisfies
$$F[\phi]=\frac 1{4\sqrt x}\geq\frac14\geq\frac18\phi\quad\text{in }(0,1),$$
whence $\ol\lambda_1(F,\O)\geq1/8$. 
\end{proof}

\begin{remark} The two operators used as examples  for case (iii) do not
satisfy hypothesis (H4), hence \thm{main} does not apply to them. Nevertheless,
they do not violate the conclusion of the theorem because, as one can
check, $\m=0$ in
both cases, independently of the extension of $F$ outside $\O$. 
This seems to suggest that hypothesis (H4) in \thm{main} could be relaxed.
\end{remark}

Proposition \ref{pro:ce} involves linear operators,
for which the notion of viscosity solution could appear artificial.
However, one cannot characterize the validity of the
Maximum Principle for $C^2$ solutions in terms of the signs of $\l$,
$\ol\lambda_1$,
$\mu_1$ or $\lambda_*$. Indeed any $C^2$ (or even $C^0$) subsolution of the
equation $F[u]:=x^2 u=0$ in $\O=(-1,1)$ is necessarily nonpositive, but it is
not hard to check that
$\l(F,\O)=\ol\lambda_1(F,\O)=\mu_1(F,\O)=\lambda_*(F,\O)=0$.
Also, notice that the operators used in the proof of Proposition \ref{pro:ce}
would still yield the result under the additional requirement that 
$\phi\in C^2(\O)$ in the definitions of $\l$ and $\ol\lambda_1$.
\\

The case (ii) in Proposition \ref{pro:ce} shows that, for degenerate
elliptic operators, the notion of generalized \pe\ is unstable with respect to
perturbations of the operator. Thus, owing to \thm{main}, the same is in some
sense true for the \MP\ property.
We now present an example that exhibits the instability of the notions $\l$,
$\ol\lambda_1$, $\mu_1$ with respect to perturbations of
the operator and approximations of the domain from inside. Let $F$ be the
operator defined by $F[u]=-xu'$, $x\in\O=(0,1)$. It turns out that
$$\l(F,\O)=\ol\lambda_1(F,\O)=\mu_1(F,\O)=0,$$
$$\forall\e>0,\quad
\l(F-\e D,\O)=\ol\lambda_1(F-\e D,\O)=\mu_1(F-\e D,\O)=+\infty,$$
$$\forall\O'\subset\subset\O,\quad
\l(F,\O')=\ol\lambda_1(F,\O')=\mu_1(F,\O')=+\infty.$$

The instability of the \pe\ is one of the main differences with the uniformly
elliptic case. In particular, the stability with respect to interior
perturbations of the domain is crucial in the arguments of \cite{BNV}. Its
validity is based on the Harnack inequality, which is not available in the
general degenerate elliptic case.

It is straightforward to check that $\mu_1$ is stable with respect to perturbations
of the domain from
outside. If the same property holds for $\l$, $\ol\lambda_1$ then they
coincide with
$\mu_1$. Proposition \ref{pro:ce} shows that this is not always the case.



\section{Proof of \thm{main}}\label{sec:proof}

We start by proving that the condition $\mu_1(F,\O)>0$ is sufficient for the
\MP\ property to hold.

\begin{proposition}\label{pro:CS}
If (H1)-(H4) hold and $\mu_1(F,\O)>0$ then $F$ satisfies \MP\
in $\O$.
\end{proposition}

\begin{proof}
If $\mu_1(F,\O)>0$ then there exist $\lambda>0$, $\O'\supset\ol\O$ and $\phi\in
LSC(\O')$ such that
$$\phi>0\text,\quad F[\phi]-\lambda\phi^\alpha\geq0\quad\text{in }\O'.$$
Up to shrinking $\O'$, it is not restrictive to assume that $\phi\in
LSC(\ol{\O'})$ and $\phi>0$ in $\ol{\O'}$.
Assume by contradiction that \eq{DP} admits a subsolution $u$ which is positive
somewhere in $\ol\O$. We claim that the function $\t u$ defined by
$$\t u(x):=\begin{cases} \max(u(x),0) & \text{if }x\in\ol\O\\
            0 & \text{otherwise},
           \end{cases}$$
satisfies $F[\t u]\leq0$ in $\O'$. Indeed, if $\psi$ is a
smooth function touching $\t u$ from above at some $x_0\in\O'$,
then either $\t u(x_0)=0$, or $\t u(x_0)=u(x_0)>0$ and $x_0\in\ol\O$. In the
first case $\psi$ has a local minimum at $x_0$ and then $F[\psi](x_0)\leq0$ by
(H1) and (H2),
in the second case $F[\psi](x_0)\leq0$ because $u$ is a subsolution of \eq{DP}.
Next, up to replacing $\phi$ with
$\left(\max_{\ol\O}\frac{\t u}\phi\right)\phi$, we can restrict the study to
the case where
$\max_{\ol{\O'}}(\t u-\phi)=0$.
Then, the standard doubling variable technique used to prove the comparison principle
yields a contradiction (see Theorem 3.3 in \cite{user}). Let us sketch the
argument. Define the following function on $\ol{\O'}\times\ol{\O'}$:
$$\Phi(x,y):=\t u(x)-\phi(y)-\frac n2|x-y|^2.$$
Calling $(x_n,y_n)$ a maximum point for $\Phi$ in
$\ol{\O'}\times\ol{\O'}$, we see that
$$0=\max_{x\in\ol{\O'}}\Phi(x,x)\leq\Phi(x_n,y_n)=\t u(x_n)-\phi(y_n)
-\frac n2|x_n-y_n|^2.$$
It follows that $x_n-y_n=o(1)$ as $n\to\infty$. Whence, since 
$\t u(x_n)-\phi(y_n)\geq0$, $x_n$ and
$y_n$ converge (up to subsequences) to a point $z$ where $\t u-\phi$
vanishes and $|x_n-y_n|^2=o(n^{-1})$. In particular, $z\in\ol\O$. We can
therefore apply Theorem 3.2 of \cite{user} and find that
$$
F(y_n, \phi(y_n),n(x_n-y_n),Y)-F(y_n,\t u(x_n),n(x_n-y_n), X)\geq
\lambda\phi^\alpha(y_n),
$$
for some $X,Y\in\mc{S}_{N}$ satisfying
$$
-3n
\left(\begin{matrix}
 I & 0 \\ 0 & I
\end{matrix}\right)
 \leq
\left(\begin{matrix}
 X & 0    \\ 0 & -Y
\end{matrix}\right)
 \leq 3n
\left(\begin{matrix}
 I & -I    \\ -I & I
\end{matrix}\right).
 $$
Since, as $n\to\infty$,  $\t u(x_n)-\phi(y_n)=o(1)$, using (H3), (H4) we
eventually derive
$$\lambda\phi^\alpha(y_n)\leq o(1)+\omega(n|x_n-y_n|^2+|x_n-y_n|).
$$
That is, $\phi(z)\leq0$, which is a contradiction.
\end{proof} 

Let us prove now that if $F$ satisfies \MP\ in $\O$ then $\mu_1(F,\O)>0$. 
This is a consequence of the following general property of $\mu_1$.

\begin{proposition}\label{pro:U}
Under the assumptions (H1)-(H4), there exists a nonnegative subsolution $U\in
USC(\oo)$,
$U\not\equiv0$, of the problem
$$
\begin{cases}
F[U]-\mu_1(F,\O) U^\alpha\leq0 &\hbox{in $\O$,} \\
U\leq 0 & \hbox{on $\po$}.
\end{cases}
$$
\end{proposition}

\begin{proof}
We construct the subsolution $U$ at the eigenlevel $\mu_1(F,\O)$ following
the method of \cite{BD07}: we solve the problem at level less than
$\mu_1(F,\O)$ with a positive right-hand side (say equal to 1) and we show 
that as the
level approaches $\mu_1(F,\O)$ the renormalized solutions tend to a function $U$
satisfying the desired property. An extra difficulty with respect to \cite{BD07}
is that $U$ could be positive somewhere on $\partial\O$, due to the lack of
existence of barriers. In order to show that $U\leq0$ on $\partial\O$ in the
viscosity sense, we combine the above procedure with an external approximation
of the domain $\O$. This is the point where the definition of $\mu_1$ is
really exploited.


Let $(\O_n)_{n\in\N}$ be a family of smooth domains such that
$$\bigcap_{n\in\N}\ol\O_n=\ol\O,\qquad\forall
n\in\N,\quad\ol\O\subset\O_{n+1}\subset\O_n\subset\mc{O}.
$$
For $n\in\N$, we consider subsolutions of the equation
\Fi{n}
F[u]-\left(\mu_1(F,\O)-\frac1n\right)u^\alpha=1\quad\text{in } \mc{O},
\Ff
whose support is contained in $\ol\O_n$. Following Perron's method, we 
define
$$
\forall x\in \mc{O},\quad
w_n(x):= \sup\{z(x) \ :\ z\in USC(\mc{O})\text{ is a subsolution of \eq{n}},
z=0\text{ outside }\ol\O_n\}.$$
The function $w_n$ could possibly be infinite at some -and even
any- point of $\ol\O_n$. Taking $z\equiv0$ yields $w_n\geq0$.
We claim that 
\Fi{blowup}
\lim_{n\to\infty}\sup_{\O_n}w_n=+\infty.
\Ff
Assume by way of contradiction that \eq{blowup} does not hold.
Then $(w_n)_{n\in\N}$ satisfies (up to subsequences) $\sup_{\O_n}w_n\leq C$, for
some $C$ independent of $n$.
For $n\in\N$, consider the
lower and upper semicontinuous envelopes of $w_n$:
$$
\forall x\in \mc{O},\quad (w_n)_*(x):=\lim_{r\to0^+}\inf_{|y-x|<r}
w_n(y),\quad
(w_n)^*(x):=\lim_{r\to0^+}\sup_{|y-x|<r} w_n(y).
$$ 
It follows from the standard theory (see Lemma 4.2 in \cite{user}) that 
$(w_n)^*$ is a subsolution of \eq{n}. Since the function $w_n$ vanishes outside 
$\ol\O_n$, its definition yields $w_n=(w_n)^*$. By Lemma 4.4 in 
\cite{user}, 
 if 
$(w_n)_*$ fails to be a supersolution of \eq{n} at some point in $\O_n$
then there exists a subsolution of \eq{n} larger than $w_n$ and still
vanishing outside $\ol\O_n$, which contradicts the definition of $w_n$.
Therefore,
$F[(w_n)_*]-(\mu_1(F,\O)+1/n)((w_n)_*)^\alpha\geq1$ in
$\O_n$,
and clearly $0\leq(w_n)_*\leq w_n\leq C$. 
As a consequence,
$$
F[(w_n)_*]-\left(\mu_1(F,\O)+\frac1n\right)((w_n)_*)^\alpha\geq
-\frac2n((w_n)_*)^\alpha+1\geq1-\frac{2C^\alpha}n\quad\text{in }\O_n.
$$
It follows that, for $n$ large enough, $(w_n)_*$ satisfies
$F[(w_n)_*]-(\mu_1(F,\O)+1/n)((w_n)_*)^\alpha>0$ in $\O_n$, and by (H3) the
same is true for $(w_n)_*+\e$, with $\e>0$ small enough.
This contradicts the definition of $\mu_1$, hence \eq{blowup} is proved.
%
There exists then a family $(z_n)_{n\in\N}$, with $z_n\in USC(\mc{O})$ 
subsolution of \eq{n} vanishing outside $\ol\O_n$, such that
$$\lim_{n\to\infty}\max_{\ol\O_n}z_n=+\infty.$$ 
Replacing $z_n$ with its
positive part, it is not restrictive to assume that $z_n\geq0$.
The functions $u_n$ defined by
$$u_n(x):=\frac{z_n(x)}{\max_{\ol\O_n}z_n},$$
satisfy
$$u_n=0 \text{ outside }\ol\O_n,\qquad\max_\mc{O}u_n=1,\qquad 
F[u_n]-\left(\mu_1(F,
\O)-\frac{1}n\right)u_n^\alpha\leq
\left(\max_{\ol\O_n}z_n\right)^{
-\alpha } \quad\text{in }\mc{O}.$$ 
Define the function $U$ by setting
$$\forall x\in\mc{O},\quad
U(x):= \lim_{j\to\infty}\sup\{u_n(y) \ :\ n\geq j,\ |x-y|<1/j\}.$$ 
By stability of viscosity subsolutions (see e.g.~Remark 6.3 in \cite{user}), we
know that $U$ satisfies
$F[U]-\mu_1(F,\O)U^\alpha\leq0$ in $\mc{O}$. 
Moreover, $U=0$ outside $\ol\O$ and $\max_{\ol\O}U=1$. It remains to show that
$U$ satisfies the Dirichlet condition on $\partial\O$ in the relaxed viscosity
sense. Suppose that there exists $\xi\in\partial\O$, $\rho>0$ and
$\vp\in C^2(\ol\O)$ such that
$$U(\xi)>0,\qquad \sup_{\ol\O\cap B_\rho(\xi)}(U-\vp)=(U-\vp)(\xi)=0.$$
By continuity of $F$, we can assume that $\vp$ is a paraboloid, thus
defined in the whole $\R^N$. Up to decreasing $\rho$ if need be, we have that
$\vp>0$ in $B_\rho(\xi)$. Since $U=0$ outside $\ol\O$, we infer that 
$\sup_{\mc{O}\cap B_\rho(\xi)}(U-\vp)=(U-\vp)(\xi)=0$, whence
$F[\vp](\xi)-\mu_1(F,\O)\vp^\alpha(\xi)\leq0$.
\end{proof}

As a corollary, we immediately deduce that if $F$ satisfies \MP\ in
$\O$ then $\mu_1(F,\O)>0$. 

\begin{remark}\label{rem:pef}
The function $U$ constructed in the above proof is a good candidate for being
the principal eigenfunction of $F$ in $\O$, i.e., a
positive solution of 
\Fi{pef}\begin{cases}
F[U]-\m U^\alpha=0 & \text{in }\O\\
U=0 & \text{on }\partial\O.
\end{cases}\Ff
However, this is not true in general. There are indeed operators which
do not admit a
principal eigenfunction. It is clearly the case if
$\m=+\infty$, as for instance for the operator $F[u]=u'$. 
An example with $\m$ finite is given by the operator $F[u]=x^2 u'$ in
$\O=(-1,1)$. Indeed, the
indicator function of $\{0\}$ violates \MP, and then $\mu_1(F,\O)\leq0$ by
\thm{main}. On the other hand, $\mu_1(F,\O)\geq0$, as it is seen by taking
$\phi\equiv1$ in the definition. Hence, $\mu_1(F,\O)=0$. But the unique solution
of \eq{pef} is $U\equiv0$.
\end{remark}


\section{Conditions for the equivalence between $\mu_1$ and $\ol\lambda_1$}

\thm{main} provides a characterization of the \MP\ property in terms of the sign
of the
generalized \pe\ $\mu_1$. We do not know if $\mu_1$ can be replaced by the more
intrinsic notion $\ol\lambda_1$, that is, if $\mu_1$ and $\ol\lambda_1$
always have the same
sign.
This property reduces to the equivalence of $\mu_1$ and $\ol\lambda_1$,
because they
satisfy
$$\forall \lambda\in\R,\quad
\ol\lambda_1(F[u]+\lambda u^\alpha,\O)=\ol\lambda_1(F,\O)+\lambda,\quad
\mu_1(F[u]+\lambda u^\alpha,\O)=\mu_1(F,\O)+\lambda.$$

Let us see what happens if we try to follow the arguments in the proof of
Proposition \ref{pro:CS} with $\m$ replaced by $\ol\lambda_1(F,\O)$. The
difference
is that now the supersolution $\phi$ is only defined in $\O$, but still has
positive infimum. Setting
$\phi(\xi):=\liminf_{x\to\xi}\phi(x)$ for $\xi\in\partial\O$,
one sees that the arguments fail only if the points $y_n$ used in the
proof belong to 
$\partial\O$. 
This difficulty can be overcome if at any $\xi\in\partial\O$ one of the
following occurs:
\Fi{Fichera-si} 
\text{any subsolution $u\in USC(\ol\O)$ of \eq{DP} satisfies $u(\xi)\leq0$,}
\Ff
\Fi{Fichera-no}
\begin{array}{c}
\text{any strictly positive supersolution $\phi$ of
$F[\phi]=\lambda\phi^\alpha$ in $\O$, $\lambda>0$,}\\
\text{is a supersolution in $\O\cup\Gamma$, for some neighbourhood $\Gamma$ of 
$\xi$.}
\end{array}
\Ff
Indeed, the limit $\xi$ of (a subsequence of) $y_n$ cannot satisfy
\eq{Fichera-si},
but if \eq{Fichera-no} holds one can conclude exactly as in  the 
the proof of Proposition \ref{pro:CS}.

This Section is devoted 
to establish sufficient conditions for either \eq{Fichera-si} or \eq{Fichera-no} to occurr, in order  to have $\mu_1=\ol \lambda_1$.
Under suitable assumptions on $F$, the case \eq{Fichera-si} is guaranteed by the
existence of a continuous {\em barrier} (see Definition \ref{def:barrier}
below).
This is shown in Section \ref{sec:barriers}. In Section \ref{sec:linear}  we show that, for linear operators,  $\mu_1=\ol\lambda_1$  if the boundary only contains connected components where the so--called Fichera condition is satisfied or violated.


\subsection{Problems with barriers}\label{sec:barriers}

Here is the definition of barrier.
\begin{definition}\label{def:barrier}
We say that a point $\xi\in\partial\O$ admits a (continuous)
{\em barrier} if there exists a ball $B$ centred at $\xi$
and a nonnegative function $w\in C(\ol{\O\cap B})$ vanishing at $\xi$ and
satisfying $F[w]\geq1$ in $\O\cap B$. 
\end{definition}
We will need the following extra assumptions on $F$ in an open neighbourhood
$V$ of $\partial\O$:
\begin{itemize}
  \item[(H5)]\label{continuity}
For all $R>0$, $(r,p,X)\mapsto F(x,r,p,X)$ is uniformly continuous in
$[0,R]\times\R^N\times\mc{S}_N$, uniformly with respect to $x\in\ol\O\cap V$.
 \item[(H6)]\label{somma'}
For all $R>0$, there exists $K>0$ such that
if $X,Y\in\mc{S}_N$ satisfy
$$\exists\sigma>0,\quad
-\sigma\left(\begin{matrix}
               I & 0 \\
0 & I
              \end{matrix}
\right)
\leq\left(\begin{matrix}
               X & 0 \\
0 & -Y
              \end{matrix}
\right)\leq\sigma\left(\begin{matrix}
               I & -I \\
-I & I
              \end{matrix}
\right),$$
then
$$F(y,r,p,Y)-F(x,r,p,X)\leq K(1+|x-y||p|+\sigma|x-y|^2),
\quad\forall x,y\in\ol\O\cap V,\ |r|\leq R,\ p\in\R^N.$$
\end{itemize}

\begin{remark}
Condition (H5) implies that the degree of homogeneity $\alpha$ in (H2) must be  less
than or equal to $1$.  

Overall, conditions (H4)--(H6) (or close variations) are often required in the context of comparison  of viscosity solutions possibly discontinuous at the boundary, see e.g. assumptions (7.15)--(7.16) in \cite{user}, Section 7.
\end{remark}

%
\begin{theorem}\label{thm:l''=mu}
If (H1), (H2), (H4)-(H6) hold, $\O$ is smooth and every point
$\xi\in\partial \O$
admits a barrier, then $\m=\ol\lambda_1(F,\O)$.
\end{theorem}
As explained before, in order to prove the result it is sufficient to show that
\eq{Fichera-si} holds at every $\xi\in\partial\O$.
\thm{l''=mu} is then a consequence of the following
%
%
%
\begin{proposition}\label{pro:barrier}
Assume that $F$ satisfies (H1), (H2), (H5), (H6), that $\O$ is a smooth domain
and that there exists $\xi\in \partial \O$
admitting a barrier. Then every subsolution $u\in USC(\ol\O)$ of
\eq{DP} satisfies $u(\xi)\leq 0$. 
%
\end{proposition}

\begin{proof}
Let $w\in C(\ol{\O\cap B})$ be the barrier at $\xi$, provided by Definition
\ref{def:barrier}. Conditions (H2), (H5) imply that, up to replacing $w$ with
$2w+k|x-\xi|^2$, with $k>0$ small enough,
it is not restrictive to assume that $w>0$ outside the point $\xi$.
We can also suppose without loss of generality that
$w\geq1>u$ on $\ol\O\cap\partial B$. Assume by contradiction that
$u(\xi)>0$. For $\e>0$, we set
$$w_\e=w+\vep,\qquad
k_\vep:= \max\limits_{\overline{B\cap \Omega}} \,\,\frac{u}{w_\e}\,.
$$
Let $x_\vep$ be a  point where $k_\vep$ is attained. Since $k_\vep \geq
\frac{u(\xi)}\vep$, we have that, as $\e\to0^+$, $k_\vep\to \infty$, whence
$x_\e\to\xi$. 
Then, it makes sense to use $\nu(x_\vep):=D d(x_\vep)$, where
$d(x)$ is the signed distance function from $\partial\O$, positive inside $\O$
and smooth in a neighbourhood of $\partial \O$.  We  follow now the
strategy of the strong comparison principle when comparing a continuous
supersolution with a possibly discontinuous subsolution, see Theorem
7.9 in \cite{user}.  We consider the function
$$
\Phi(x,y)= u(x)- k_\vep w_\e(y)- |n(x-y)+\de \nu(x_\vep)|^2
-\de|x-x_\vep|^2\qquad x,y\in B_\rho\cap \ol \Omega\,,
$$
where $n$, $\de>0$. Let then $(x_n, y_n)\in \ol\O$ be such that
$$
\Phi(x_n,y_n)= \max_{\ol{B\cap\O}}\Phi(x,y)\,.
$$
Of course the two points also depend on $\de$, $\vep$ but we avoid to stress
this fact to simplify the notation.  We have that $\Phi(x_n,y_n)\geq
\Phi(x_\vep, x_\vep)=-\de^2$.
Furthermore, for $n$ large, $x_\vep+\frac \de n \nu(x_\vep)\in B\cap
\Omega$, hence $\Phi(x_n,y_n)\geq \Phi(x_\vep, x_\vep+\frac \de n \nu(x_\vep))$,
which implies
$$
|n(x_n-y_n)+\de \nu(x_\vep)|^2+ \de|x_n-x_\vep|^2 \leq u(x_n)- k_\vep w_\e(y_n)
- u(x_\vep) + k_\vep w_\e(x_\vep + \frac \de n \nu(x_\vep))\,.
$$
Since $w$ is continuous, we have $w_\e(x_\vep + \frac \de n \nu(x_\vep))=
w_\e(x_\vep) + o(1)$ as $n\to \infty$, hence, using also that
$k_\e w_\e(x_\vep)=u(x_\e)$, we deduce
$$
|n(x_n-y_n)+\de \nu(x_\vep)|^2+ \de|x_n-x_\vep|^2 \leq u(x_n)- k_\vep w_\e(y_n)+
o(1) \quad \hbox{as $n\to \infty$.}
$$
We first use this inequality to infer that
both $x_n$ and $y_n$ converge to $x_\vep$ as $n$ tends to infinity. Then,
together with the upper semicontinuity of $u$ and the fact that $u-k_\e w\leq0$,
it implies that $n(x_n-y_n)+\de \nu(x_\vep)=o(1)$ as $n\to
\infty$. Since $\nu$ is continuous we eventually derive
$$
y_n= x_n+ \frac \de n \nu(x_n) + o\left( \frac 1n\right)\qquad \hbox{ as $n\to \infty$.}
$$
It follows that $y_n\in \Omega$ for $n$ large. This allows us to
use the
equation of $w$ as a
supersolution. As far as $u$ is concerned, we have
$$
u(x_n)\geq \Phi(x_n,y_n)+ k_\vep w_\e(y_n)\geq \vep k_\vep - \de^2\,.
$$
Choosing $\de$ small enough, compared to $\vep k_\vep$, we get that $u(x_n)>0$
so that we can use the equation of $u$ at $x_n$ even if
$x_n\in \partial \O$. Usual viscosity arguments (see Theorem 3.2 in 
\cite{user}) yield
\Fi{>k}
F(y_n, k_\vep w_\e(y_n), q, Y)-F(x_n, u(x_n), p, X)\geq k_\vep^\alpha\,,
\Ff
where $p=2n(n(x_n-y_n)+\de \nu(x_\vep))
+2\delta(x_n-x_\vep)$, $q=
2n(n(x_n-y_n)+\de \nu(x_\vep))$, $X$, $Y$ satisfy
$$
-  (2n^2+\| A\|)
\left(\begin{matrix}
 I & 0 \\ 0 & I
\end{matrix}\right)
 \leq
\begin{pmatrix}
 X & 0    \\
\noalign{\medskip}
0 & -Y
\end{pmatrix}
 \leq A+
 \frac1{2n^2} \,A^2,\qquad A= \begin{pmatrix}
 2n^2I+2\de I & -2n^2 I   \\
\noalign{\medskip}
-2n^2 I & 2n^2I
\end{pmatrix}.
 $$
Since
$$\|A\|\leq4n^2+2\delta,\qquad 
A^2\leq4n^2(2n^2+\delta)\left(\begin{matrix}
                               I & -I \\ -I & I
                              \end{matrix}
\right)+4\delta(\delta+n^2)\left(\begin{matrix}
                               I & 0 \\ 0 & I
                              \end{matrix}
\right),$$
we derive
$$-(6n^2+2\delta+\beta)\left(\begin{matrix}
 I & 0 \\ 0 & I
\end{matrix}\right)
 \leq
\begin{pmatrix}
 X-\beta I & 0    \\
\noalign{\medskip}
0 & -(Y+\beta I)
\end{pmatrix}
 \leq (6n^2+2\delta)
\left(\begin{matrix}
   I & -I \\ -I & I
\end{matrix}\right),\quad \beta=4\delta+\frac{2\delta^2}{n^2}.$$
Hence, by (H6), there is $K$ such that, for bounded $r$,
$$F(y_n, r, q, Y+\beta I)-F(x_n, r, q, X-\beta I)\leq K
[1+(6n^2+2\delta+\beta)|x_n-y_n|^2+|x_n-y_n|(|q|+1)].$$
Since $-\delta^2\leq u(x_n)-k_\e w_\e(y_n)\leq o(1)$ as $n\to\infty$, by (H5) we
can choose $\delta$ small enough and  $n$ large in such a way that
$$F(y_n,k_\e w_\e(y_n),q,Y)-F(x_n,u(x_n),p,X)\leq2K.$$
Whence, by \eq{>k}, $k_\e^\alpha\leq2K$, which is impossible since
$k_\vep\to \infty$ as $\e\to0^+$.
%
\end{proof}

%

\begin{remark}\rm  At least if $F$ is linear, the existence of  a global smooth
barrier implies that $\mu_1=\la_1$. 

Namely, assume that there exists $v\in C^2$ such that $F[v]\geq 1$ in some 
neighbourhood of $\partial \O$ and $v=0$  on $\partial \O$.  
Let $\la>0$ be such that  $F[\vfi]\geq \la \vfi$ for some $\vfi\in LSC(\O)\cap
L^\infty(\O)$ such that $\vfi>0$ in $\O$. There exists $\rho>0$ (only depending
on $\la$) such that $F[v] \geq \la v+\frac12$ if $d(x)<\rho$. Let us take any
small $\vep>0$; if $\zeta$ is a cut--off function such that $\zeta=1$ if
$d(x)<\rho$ and $\zeta=0$ if $d(x)\geq 2\rho$, we claim that  $w:=\vfi+
\de(v+\vep)\zeta$ is a supersolution of $F[w]\geq (\la -\vep) w$ in $\Omega$ up
to choosing a  suitable $\de$. Indeed, if $d(x)\geq 2\rho$, since $\zeta=0$ we
have
$$
F[w]= F[\vfi] \geq \la \vfi=\la w,
$$
whereas, if $d(x)<\rho$, since $\zeta=1$ we see that
$$
F[w] \geq F[\vfi]+ F[\de(v+\vep )]\geq \la \vfi+ \de \la v+\frac12\de+ c(x) \de
\vep \geq \la w +\de( \frac12- (\la+ |c(x)|)  \vep)  \geq \la w 
$$
if $\vep$ is small.
In the set $\{\rho<d(x)<2\rho\}$ we have
$F[\de(v+\vep)\zeta]\geq - C(v, \zeta) \de$ and $\inf \vfi>0$, whence
$$
F[w] \geq \la \vfi - C(v,\zeta)\de\geq (\la-\vep)w+ \vep\vfi - \tilde C \de \geq (\la-\vep) w
$$
provided $\de$ is sufficiently small. Finally, for any $\vep>0$ we can find $\de$   such that
$F[w]\geq (\la-\vep) w
$ in $\O$, and since $w>0$ in $\ol\O$ we deduce that $\la_1''(\O)\geq \la-\vep$. Since $\la$ was any value smaller than $\la_1$ and $\vep$ is arbitrary, we conclude that $\la_1''\geq \la_1$, and therefore $\la_1''=\la_1=\mu_1$.
\end{remark}


\subsection{Linear operators}\label{sec:linear}

In this section $F$ is a degenerate elliptic linear
operator. Namely,
$$F[u]=-\Tr(A(x)D^2u)-b(x)\.Du-c(x)u,\quad x\in\mc{O},$$
with $A=\Sigma^t\Sigma$, $\Sigma:\ol\O\to\mc{S}_N$, $b:\ol\O\to\R^N$ and
$c:\ol\O\to\R$.
We will require that
\Fi{hyp:linear}
\Sigma,b\in W^{1,\infty}(\ol\O),\qquad c\in C(\ol\O),
\Ff
As shown in Example 3.6 of \cite{user}, the Lipschitz continuity of $\Sigma$
is precisely the condition for the second order term to satisfy (H4). The
Lipschitz continuity of $b$ could be relaxed by 
\Fi{b}
\exists K>0,\quad\forall x,y\in\ol\O,\quad (b(x)-b(y))\.(x-y)\geq-K|x-y|^2
\Ff
in order to fulfil (H4), but $b\in W^{1,\infty}(\ol\O)$ is needed to have 
(H6). 

We say that the {\em Fichera condition} is satisfied
at a point
$\xi\in\partial\O$ if one of the following two cases occurs:
$$Dd(\xi)A(\xi)Dd(\xi)>0,\quad\text{or }\quad
\begin{cases}
Dd(\xi)A(\xi)Dd(\xi)=0\\ 
\Tr(A(\xi)D^2d(\xi))+ b(\xi)\cdot Dd(\xi) <0,
\end{cases}$$
where, as before, $d$ is the signed distance function from $\partial\O$,
positive inside $\O$. This condition was introduced by G.~Fichera in
\cite{Fichera}  in order to study the question whether the Dirichlet condition should be assumed or not at boundary points. See also \cite{BarBur} for a discussion of the same problem in terms of viscosity solutions.

\begin{theorem}\label{thm:l''=mulinear}
If \eq{hyp:linear} holds, $\O$ is smooth and in every connected component of
$\partial\O$ the Fichera condition is either always satisfied or always
violated, then $\mu_1(F,\O)=\ol\lambda_1(F,\O)$.
\end{theorem}

\begin{remark} The previous result applies to a  significant example, namely to the case that the domain $\Omega$ is invariant for the associated stochastic dynamics $dX_t = b(X_t)dt + \sqrt 2\, dW_t$ defined in  a standard probability space, being $W_t$ a Wiener process in $\R^N$. 
In fact, it is well known that $\Omega$ is invariant (and, at the same time, $\ol\Omega$ is invariant) if and only if the Fichera condition is  
violated everywhere on the boundary, see e.g. \cite{Fried}, \cite{Fried-Pinsky} and \cite{CDF} for a  complete discussion of this property even in non smooth domains. 
\end{remark}

Recall that the result would follow if we show that
\eq{Fichera-si} or \eq{Fichera-no} hold at every $\xi\in\partial\O$.
One can readily check that the Fichera condition implies that, for
$\delta>0$ small enough, the function 
$w(x):=\log(\delta+d(x))-\log\delta$ is a barrier at $\xi$ in the sense of
Definition \ref{def:barrier}. Thus, by Proposition \ref{pro:barrier},
\eq{Fichera-si} holds in the connected components where the Fichera condition
is fulfilled. Let us show that \eq{Fichera-no} holds in the others. Since the
Fichera condition does not involve the zero order term of the operator, we can
restrict to $\lambda=0$ in \eq{Fichera-no}.
Hence, the proof of \thm{l''=mulinear} relies on the following result,   which is essentially proved in \cite{Barles-Rouy}, Lemma 4.1. For the sake of clarity,  since there are minor differences in our setting, we provide  a simple proof below.

\begin{lemma}\label{superat0}
Assume that \eq{hyp:linear} holds, $\O$ is smooth and the Fichera condition 
is not
satisfied in an open subset $\Gamma$ of $\partial\O$ (in the induced topology),
that is,
$$\forall\xi\in\Gamma,\quad Dd(\xi)A(\xi)Dd(\xi)=0,\quad\Tr(A(\xi)D^2
d(\xi))+ b(\xi)\cdot Dd(\xi)\geq0.$$
Then, any supersolution
$\phi\in LSC(\O)$ of $F=0\text{ in }\O$, which is bounded from below, extended
to $\Gamma$ by setting
$$\forall\xi\in\Gamma,\quad\phi(\xi):=\liminf_{\su{x\to \xi}{x\in\O}}\phi(x),$$
is a supersolution in
$\O\cup\Gamma$.
\end{lemma}

\begin{proof} 
In this statement, we use the convention that $\phi$ automatically satisfies
the condition of being a supersolution at the points $\xi\in\Gamma$ where
$\phi(\xi)=+\infty$.
%
%
%
Let $\xi\in\Gamma$ and $\psi\in C^2(\O\cup\Gamma)$
be such that $(\phi-\psi)(\xi)=\min_{\ol\O\cap B}(\phi-\psi)=0$, for some
closed ball $B$ (with positive radius) centered at $\xi$ satisfying 
$B\cap\partial\O\subset\Gamma$.
Our aim is to show that $F[\psi](\xi)\geq0$.
By usual arguments, it is not restrictive to assume that the above minimum is
strict.
Consider the family of functions $(\psi_\e)_{\e>0}$ defined in $\O$ by
$\psi_\e(x):=\psi(x)+\e\log(d(x))$. Let $(x_\e)_{\e>0}$ in $\O\cap B$ be such
$(\phi-\psi_\e)(x_\e)=\min_{\O\cap B}(\phi-\psi_\e)$, and let
$\zeta\in\ol\O\cap B$ be the limit as $\e\to0^+$ of (a subsequence of) $x_\e$.
For $x\in\O\cap B$, we see that
\[\begin{split}
(\phi-\psi)(x)&=\lim_{\e\to0^+}(\phi-\psi_\e)(x)\geq
\liminf_{\e\to0^+}(\phi-\psi_\e)(x_\e)\geq(\phi-\psi)(\zeta)-
\limsup_{\e\to0^+}\e\log(d(x_\e))\\
&\geq(\phi-\psi)(\zeta).
\end{split}\]
Since this holds for any $x\in \O\cap B$, applying this inequality to a sequence of points along which 
$\phi$ tends to $\phi(\xi)$, we infer that $\zeta=\xi$, because $\phi-\psi$
has a strict minimum at $\xi$. This shows that $x_\vep\to\xi$ as $\vep\to
0^+$. 
In particular, since $x_\vep\notin\partial B$, 
we deduce,
being $\phi$ a supersolution in $\O$, that
$$
[-\Tr(AD^2\psi)-b\cdot D\psi-c\phi
-\e(d^{-1}\Tr(AD^2d)-d^{-2}DdADd+
d^{-1}b\cdot Dd)](x_\e)\geq0.
$$
This inequality  reads as
$$[F[\psi]+c(\psi-\phi)-\e(d^{-1}\Tr(AD^2d)-d^{-2}DdADd+
d^{-1}b\cdot Dd)](x_\e)\geq0.
$$
For $\e$ small enough, $x_\e$ has a unique projection $\xi_\e$ on $\partial\O$,
the function $d$ is smooth in a neighbourhood of $x_\e$ and satisfies
$Dd(x_\e)=(x_\e-\xi_\e)/|x_\e-\xi_\e|=:-\nu(\xi_\e)$.
Up to decreasing $\e$, we have that $\xi_\e\in\Gamma$ because $x_\e,\xi_\e\to\xi$
as $\e\to0$. It follows that $\Sigma(\xi_\e)\nu(\xi_\e)=0$ and thus
$DdADd(x_\e)=\|\Sigma(x_\e)\nu(\xi_\e)\|^2\geq-l^2d^2(x_\e)$,
where $l$ is the Lipschitz constant of $\Sigma$.
On the other hand, the Lipschitz continuity of $\Sigma$ and $b$ 
imply the existence of a constant $C$ such that
$$(\Tr(AD^2d)+ b\cdot Dd)(x_\e)\geq
(\Tr(AD^2d)+ b\cdot Dd)(\xi_\e)-Cd(x_\e).$$
Since $\xi_\e \in \Gamma$, we have that $(\Tr(AD^2d)+ b\cdot Dd)(\xi_\e)\geq0$ and then,
using the above inequalities, we obtain
\Fi{preps}
F[\psi](x_\vep)\geq-\e(C+l^2)+\sup_\O|c|(\phi-\psi)(x_\e).
\Ff
Since $x_\vep$ is a minimum point for $\phi-\psi_\e$, we have that
$$
\forall x\in\O\cap B,\quad
\phi(x_\vep)- [\psi(x_\vep)+\vep\log(d(x_\vep))]
\leq \phi(x)- [\psi(x)+\vep\log(d(x))].
$$
Notice that $\log(d(x_\vep))<0$ for $\e$ small enough, whence
$$
\forall x\in\O\cap B,\quad
\limsup\limits_{\vep \to0^+}(\phi-\psi)(x_\vep)\leq(\phi- \psi)(x).
$$
Choosing in place of  $x$ a sequence of points converging to $\xi$, along which $\phi$ tends
to $\phi(\xi)$,
we eventually infer that $(\phi-\psi)(x_\e)\to0$ as $\e\to0^+$.
Therefore, passing to the limit in \eq{preps} we deduce $F[\psi](\xi)\geq
0$, which concludes the proof.
\end{proof}


\begin{remark}
We do not know whether or not $\mu_1$ and $\ol\lambda_1$ do coincide when
$\partial\O$
has a connected component containing both points where the Fichera condition
is satisfied and points where it is not. 
The problem is that positive supersolutions in $\O$ may not be supersolutions
at the points $\xi$ that satisfy the Fichera condition but belong to the 
boundary of the set where the Fichera condition does not hold. In such case,
one could replace the perturbation $\e\log(d(x))$ used in the proof of Lemma
\ref{superat0} with $\e\log(|x-\xi|)$, and the perturbation terms could be controlled
if the sequences $\seq{x}$ converging to $\xi$ on which $\phi$ tends
to $\phi(\xi)$ satisfy $d(x_n,\partial \O) \gtrsim |x_n-\xi|$.
This is the so-called \emph{cone condition},  namely that the value of $\varphi$ at $\partial \Omega$ may be reached along at least one sequence of points lying in a  cone. The relevance of this condition for strong comparison results (i.e. comparison of viscosity solutions discontinuous at the boundary) was already pointed out before and specifically in connection with stochastic control problems, see \cite{Kat}, \cite{Barles-Rouy}. In particular, the conclusion of Lemma \ref{superat0} would still hold if the
cone condition is fulfilled at any point of the boundary (or at least at those
points where a  barrier does not exist).
\end{remark}

\end{document}